\documentclass[12pt]{article}

\usepackage[margin=1in]{geometry}  

\usepackage{amsthm, amssymb}
\usepackage{amsmath}               
\usepackage{amsfonts}              
\usepackage{amsthm}                
\usepackage{mathrsfs}
\usepackage{xcolor}
\usepackage[pdftex,bookmarks,colorlinks,breaklinks]{hyperref}  
\hypersetup{linkcolor=red,citecolor=blue,filecolor=dullmagenta,urlcolor=blue} 

\usepackage{verbatim} 
\usepackage{xspace}
\usepackage{graphicx}  
\usepackage[small]{caption} 
\usepackage{color}
\newcommand{\breakingcomma}{%
  \begingroup\lccode`~=`,
  \lowercase{\endgroup\expandafter\def\expandafter~\expandafter{~\penalty0 }}}
\usepackage{float}
\usepackage{wrapfig}

\usepackage{titlesec}
\titleformat{\section}{\large\bfseries}{\thesection}{.5em}{}

\newtheorem{theorem}{Theorem}

\newtheorem{lemma}[theorem]{Lemma}

\newtheorem{example}[theorem]{Example}

\begin{document}

\thispagestyle{empty}
\begin{center}
\textbf{\Large Multidesigns for a graph pair of order 6}

\addvspace{\bigskipamount}
Yizhe Gao \hspace{.5in} Dan Roberts\footnote{drobert1@iwu.edu}  \\
Department of Mathematics\\
Illinois Wesleyan University \\ Bloomington, IL 61701 
\end{center}

\begin{abstract}
Given two graphs $G$ and $H$, a $(G,H)$-multidecomposition of $K_{n}$ is a partition of the edges of $K_{n}$ into copies of $G$ and $H$ such that at least one copy of each is used.  We give necessary and sufficient conditions for the existence of $(C_{6},\overline{C}_{6})$-multidecomposition of $K_{n}$ where $C_{6}$ denotes a cycle of length 6 and $\overline{C}_{6}$ denotes the complement of $C_{6}$.  We also characterize the cardinalities of leaves and paddings of maximum $(C_{6},\overline{C}_{6})$-multipackings and minimum $(C_{6},\overline{C}_{6})$-multicoverings, repsectively.
\end{abstract}

\section{Introduction}
Let $G$ and $H$ be graphs. Denote the vertex set of $G$ by $V(G)$ and the edge set of $G$ by $E(G)$.  A \emph{$G$-decomposition of $H$} is a partition of $E(H)$ into a set of edge-disjoint subgraphs of $H$ each of which are isomorphic to $G$.  Graph decompositions have been extensively studied.  This is particularly true for the case where $H\cong K_{n}$, see \cite{Adamsetal} for a recent survey.  As an extension of a graph decomposition we can permit more than one graph, up to isomorphism, to appear in the partition.  A \emph{$(G,H)$-multidecomposition of $K_{n}$} is a partition of $E(K_{n})$ into a set of edge-disjoint subgraphs each of which is isomorphic to either $G$ or $H$, and at least one copy of $G$ and one copy of $H$ are elements of the partition.  When a $(G,H)$-multidecomposition of $K_{n}$ does not exist, we would like to know how ``close'' we can get.  More specifically, define a \emph{$(G,H)$-multipacking of $K_{n}$} to be a collection of edge-disjoint subgraphs of $K_{n}$ each of which is isomorphic to either $G$ or $H$ such that at least one copy of each is present.  The set of edges in $K_{n}$ that are not used as copies of either $G$ or $H$ in the $(G,H)$-multipacking is called the \emph{leave} of the $(G,H)$-multipacking.  Similarly, define a \emph{$(G,H)$-multicovering of $K_{n}$} to be a partition of the multiset of edges formed by $E(K_{n})$ where some edges may be repeated into edge-disjoint copies of $G$ and $H$ such that at least one copy of each is present.  The multiset of repeated edges is called the \emph{padding}.  A $(G,H)$-multipacking is called \emph{maximum} if its leave is of minimum cardinality, and a $(G,H)$-multicovering is called \emph{minimum} if its padding is of minimum cardinality.  

A natural way to form a pair of graphs is to use a graph and its complement.  To this end, we have the following definition which first appeared in \cite{AbueidaDaven}.  Let $G$ and $H$ be edge-disjoint, non-isomorphic, spanning subgraphs of $K_{n}$ each with no isolated vertices.  We call $(G,H)$ a \textit{graph pair of order $n$} if $E(G) \cup E(H) = E(K_{n})$.  For example, the only graph pair of order 4 is $(C_{4}, E_{2})$, where $E_{2}$ denotes the graph consisting of two disjoint edges.  Furthermore, there are exactly 5 graph pairs of order 5.  In this paper we are interested in the graph pair formed by a 6-cycle, denoted $C_{6}$, and the complement of a 6-cycle, denoted $\overline{C}_{6}$.

Necessary and sufficient conditions for multidecompositions of complete graphs into all graph pairs of orders 4 and 5 were characterized in \cite{AbueidaDaven}.  They also characterized the cardinalities of leaves and paddings of multipackings and multicoverings for the same graph pairs.  We advance those results by solving the same problems for a graph pair of order 6, namely $(C_{6},\overline{C}_{6})$.  We first address multidecompositions, then multipackings and multicoverings.  Our main results are stated in the following three theorems.

\begin{theorem} \label{main decomposition theorem}
The complete graph $K_n$ admits a $(C_6,\overline{C}_6)$-multidecomposition of $K_n$ if and only if $n \equiv 0,1\pmod{3}$ with $n\geq 6$, except $n\in\{7,9,10\}$. 
\end{theorem}

\begin{theorem} \label{main packing theorem}
For each $n\equiv 2\pmod{3}$ with $n\geq 8$, a maximum $(C_6,\overline{C}_6)$-multipacking of $K_n$ has a leave of cardinality 1.  Furthermore, a maximum $(C_6,\overline{C}_6)$-multipacking of $K_7$ has a leave of cardinality 6, and a maximum $(C_6,\overline{C}_6)$-multipacking of either $K_9$ or $K_{10}$ has a leave of cardinality 3.
\end{theorem}

\begin{theorem} \label{main covering theorem}
For each $n\equiv 2\pmod{3}$ with $n\geq 8$, a minimum $(C_6,\overline{C}_6)$-multicovering of $K_n$ has a padding of cardinality 2.  Furthermore, a minimum $(C_6,\overline{C}_6)$-multicovering of $K_7$ has a padding of cardinality 6, and a minimum $(C_6,\overline{C}_6)$-multicoveirng of either $K_9$ or $K_{10}$ has a padding of cardinality 2.
\end{theorem}

Let $G$ and $H$ be vertex-disjoint graphs.  The \emph{join of $G$ and $H$}, denoted $G\vee H$, is defined to be the graph with vertex set $V(G)\cup V(H)$ and edge set $E(G)\cup E(H)\cup\{\{u,v\}\colon\,u\in V(G),v\in V(H)\}$.  We use the shorthand notation $\bigvee_{i=1}^{t}G_{i}$ to denote $G_{1}\vee G_{2}\vee\cdots\vee G_{t}$, and when $G_{i}\cong G$ for all $1\leq i\leq t$ we write $\bigvee_{i=1}^{t}G$.  For example, $K_{12}\cong\bigvee_{i=1}^{4}K_{3}$.  

For notational convenience, let $(a,b,c,d,e,f)$ denote the copy of $C_{6}$ with vertex set $\{a,b,c,d,e,f\}$ and edge set $\{\{a,b\}, \{b,c\}, \{c,d\}, \{d,e\}, \{e,f\}, \{a,f\}\}$, as seen in Figure \ref{graphpair6}.  Let $[a,b,c;d,e,f]$ denote the copy of $\overline{C}_{6}$ with vertex set $\{a,b,c,d,e,f\}$ and edge set $$\{\{a,b\},\{b,c\},\{a,c\},\{d,e\},\{e,f\},\{d,f\},\{a,d\},\{b,e\},\{c,f\}\}.$$

\begin{figure}[h]
\centering
\includegraphics[scale=0.6]{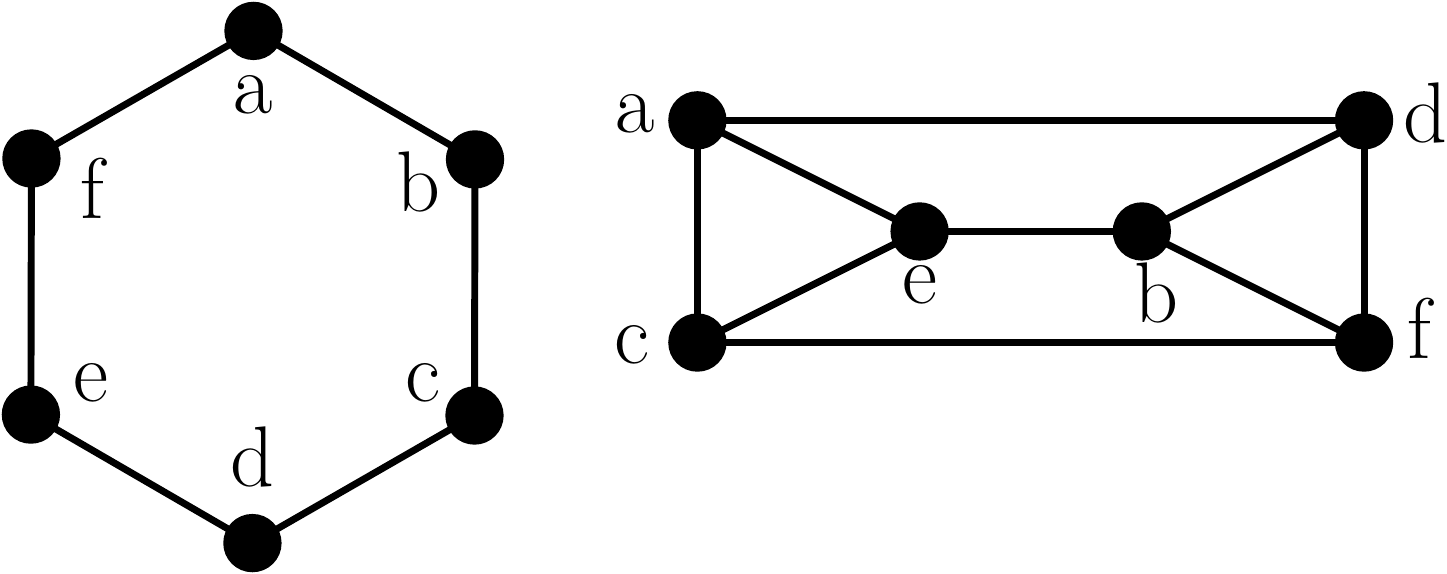}
\caption{Labeled copies of $C_{6}$ and $\overline{C}_{6}$, denoted by $(a,b,c,d,e,f)$ and $[a,e,c;d,b,f]$, respectively.}
\label{graphpair6}
\end{figure}

Next, we state some known results on graph decompositions that will help us prove our main result.  Sotteau's theorem gives necessary and sufficient conditions for complete bipartite graphs (denoted by $K_{m,n}$ when the partite sets have cardinalities $m$ and $n$) to decompose into even cycles of fixed length.  Here we state the result only for cycle length 6.  

\begin{theorem} [Sotteau \cite{Sotteau}]\label{SotteauCorollary}
A $C_{6}$-decomposition of $K_{m,n}$ exists if and only if $m \geq 4$, $n\geq 4$, $m$ and $n$ are both even, and $6$ divides $mn$.
\end{theorem}

Another celebrated result in the field of graph decompositions is that the necessary conditions for a $C_{k}$-decomposition of $K_{n}$ are also sufficient.  Here we state the result only for $k=6$.

\begin{theorem}[Alspach et al. \cite{Alspachetal}]\label{C6spectrum}
Let $n$ be a positive integer.  A $C_{6}$-decom\-position of $K_{n}$ exists if and only if $n\equiv 1, 9\pmod{12}$.
\end{theorem}

The necessary and sufficient conditions for a $\overline{C}_{6}$-decomposition of $K_{n}$ are also known, and stated in the following theorem.

\begin{theorem}[Kang et al. \cite{KangZhaoMa}]\label{ComplementSpectrum}
Let $n$ be a positive integer.  A $\overline{C}_{6}$-decomposition of $K_{n}$ exists if and only if $n\equiv 1\pmod{9}$.
\end{theorem}

\section{Multidecompositions}
We first establish the necessary conditions for a $(C_6,\overline{C}_6)$-multidecomposition of $K_{n}$.

\begin{lemma}
If a $(C_6,\overline{C}_6)$-multidecomposition of $K_{n}$ exists, then 
\begin{enumerate}
\item $n\geq 6$, and
\item $n \equiv 0,1 \pmod{3}$.
\end{enumerate}
\end{lemma}
\begin{proof}
Assume that a $(C_6,\overline{C}_6)$-multidecomposition of $K_{n}$ exists.  It is clear that condition (1) holds, with the exception of the trivial case where $n=1$.  Considering that the edges of $K_{n}$ are partitioned into subgraphs isomorphic to $C_{6}$ and $\overline{C}_{6}$, we have that there exist positive integers $x$ and $y$ such that ${n\choose 2}=6x+9y$.  Hence, 3 divides ${n\choose 2}$, which implies $n\equiv 0,1\pmod{3}$, and condition (2) follows.
\end{proof}

\subsection{Small examples of multidecompositions}
In this section we present various non-existence and existence results for $(C_6,\overline{C}_6)$-multidecompositions of small orders.  The existence results will help with our general constructions.
\subsubsection{Non-existence results}
The necessary conditions for the existence of a $(C_6,\overline{C}_6)$-multidecomposition of $K_{n}$ fail to be sufficient in exactly three cases, namely $n=7,9,10$.  We will now establish the non-existence of $(C_6,\overline{C}_6)$-multidecompositions of $K_{n}$ for these cases.

\begin{lemma}
A $(C_6,\overline{C}_6)$-multidecomposition of $K_{7}$ does not exist.
\end{lemma}
\begin{proof}
Assume the existence of a $(C_6,\overline{C}_6)$-multidecomposition of $K_{7}$, call it $\mathcal{G}$.  There must exist positive integers $x$ and $y$ such that ${7\choose 2}=21=6x+9y$.  The only solution to this equation is $(x,y)=(2,1)$; therefore, $\mathcal{G}$ must contain exactly one copy of $\overline{C}_{6}$.  However, upon examing the degree of each vertex contained in the single copy of $\overline{C}_{6}$ we see that there must exist a non-negative integer $p$ such that $6=2p+3$.  This is a contradiction.  Thus, a $(C_6,\overline{C}_6)$-multidecomposition of $K_{7}$ cannot exist.
\end{proof}

\begin{lemma}
A $(C_6,\overline{C}_6)$-multidecomposition of $K_{9}$ does not exist.
\end{lemma}
\begin{proof}
Assume the existence of a $(C_6,\overline{C}_6)$-multidecomposition of $K_{9}$, call it $\mathcal{G}$.  There must exist positive integers $x$ and $y$ such that ${9\choose 2}=36=6x+9y$.  The only solution to this equation is $(x,y)=(3,2)$; therefore, $\mathcal{G}$ must contain exactly two copies of $\overline{C}_{6}$.

Turning to the degrees of the vertices in $K_{9}$, we have that there must exist positive integers $p$ and $q$ such that $8=2p+3q$.  The only possibilities are $(p,q)\in\{(4,0),(1,2)\}$. Note that $K_{6}$ does not contain two edge-disjoint copies of $\overline{C}_{6}$.  Since $\mathcal{G}$ contains exactly two copies of $\overline{C}_{6}$, there must exist at least one vertex $a\in V(K_{9})$ that is contained in exactly one copy of $\overline{C}_{6}$. However, this contradicts the fact that vertex $a$ must be contained in either 0 or 2 copies of $\overline{C}_{6}$.  Thus, a $(C_6,\overline{C}_6)$-multidecomposition of $K_{9}$ cannot exist.
\end{proof}

\begin{lemma}
A $(C_6,\overline{C}_6)$-multidecomposition of $K_{10}$ does not exist.
\end{lemma}
\begin{proof}
Assume the existence of a $(C_6,\overline{C}_6)$-multidecomposition of $K_{10}$, call it $\mathcal{G}$.  There must exist positive integers $x$ and $y$ such that ${10\choose 2}=45=6x+9y$.  Thus, $(x,y)\in \{(6,1),(3,3)\}$; therefore, $\mathcal{G}$ must contain at least one copy of $\overline{C}_{6}$.  However, if $\mathcal{G}$ consists of exactly one copy of $\overline{C}_{6}$, then the vertices of $K_{10}$ which are not included in this copy would have odd degrees remaining after the removal of the copy of $\overline{C}_{6}$.  Thus, the case where $(x,y)=(6,1)$ is impossible.

Upon examining the degree of each vertex in $K_{10}$, we see that there must exist positive integers $p$ and $q$ such that $9=2p+3q$.  The only solutions to this equation are $(p,q)\in\{(3,1),(0,3)\}$.  From the above argument, we know that $\mathcal{G}$ contains exactly 3 copies of $\overline{C}_{6}$, say $A,B,$ and $C$.  Let $X=V(A)\cap V(B)$.  It must be the case that $|X|\geq 2$ since $K_{10}$ has 10 vertices.  It also must be the case that $|X|\leq 5$ since $K_{6}$ does not contain two copies of $\overline{C}_{6}$.  If $|X|\in\{2,3\}$, then $V(C)\cap (V(A)\bigtriangleup V(B))\neq\emptyset$, where $\bigtriangleup$ denotes the symmetric difference.  This implies that there exists a vertex in $V(K_{n})$ that is contained in exactly 2 copies of $\overline{C}_{6}$ in $\mathcal{G}$, which is a contradiction.  

Observe that any set consisting of either 4 or 5 vertices in $\overline{C}_{6}$ must induce at least 3 or 6 edges, respectively.  Furthermore, $X\subseteq V(C)$ due to the degree constraints put in place by the existence of $\mathcal{G}$.  If $|X|=4$ or $|X|=5$, then $X$ must induce at least 9 or at least 18 edges, respectively.  This is a contradiction in either case.  Thus, no $(C_6,\overline{C}_6)$-multidecomposition of $K_{10}$ exists.

\end{proof}

\subsubsection{Existence results}
We now present some multidecompositions of small orders that will be useful for our general recursive constructions. 

\begin{example}\label{K13}
$K_{13}$ admits a ($C_6$,$\overline{\textup{$C$}}_6$)-multidecomposition.
\end{example}
Let $V(K_{13})=\{1,2,\dots, 13\}$.  The following is a $(C_6,\overline{\textup{C}}_6)$-multidecomposition of $K_{13}$. 
\begin{align*}
\bigl\{ [1,2,{}&3;7,9,8], [1,4,5;9,12,10], [3,4,6;7,11,10], [2,5,6;8,12,11]\bigr\} \\
\cup \bigl\{(13,{}&1,6,8,5,11), (13,2,4,7,6,12), (13,3,5,9,4,10), (13,7,12,3,9,6), \\
&(13,8,10,2,7,5), (13,9,11,1,8,4), (1,10,3,11,2,12)\bigr\}
\end{align*}

\begin{example}\label{K15}
$K_{15}$ admits a ($C_6$,$\overline{\textup{$C$}}_6$) -multidecomposition.
\end{example}
Let $V(K_{15})=\{1,2,\dots, 15\}$.  The following is a $(C_6,\overline{\textup{C}}_6)$-multidecomposition of $K_{15}$. 
\begin{align*}
\bigl\{ [1,5,{}&10;6,8,12], [4,8,13;9,11,15], [7,11,1;12,14,3], [10,14,4;15,2,6], \\ 
& [13,2,7;3,5,9]\bigr\} \\
\cup \bigl\{(1, {}&12, 11, 13, 5, 15), (4, 15, 14, 1, 8, 3), (7, 3, 2, 4, 11, 6), (10, 6, 5, 7, 14, 9), \\
& (13, 9, 8, 10, 2, 12), (1, 2, 11, 3, 6, 13), (4, 5, 14, 6, 9, 1), (7, 8, 2, 9, 12, 4), \\
& (10, 11, 5, 12, 15, 7), (13, 14, 8, 15, 3, 10)\bigr\}
\end{align*}

\begin{example}\label{K19}
$K_{19}$ admits a $(C_{6},\overline{C}_{6})$-multidecomposition.
\end{example}
Let $V(K_{19})=\{1,2,\dots, 19\}$.  The following is a $(C_6,\overline{C}_6)$-multidecomposition of $K_{19}$. 
\begin{align*}
\bigl\{ [2, {}&11, 14; 17, 4, 18], [3, 12, 15; 18, 5, 19], [4, 13, 16; 19, 6, 11], [5, 14, 17; 11, 7, 12], \\
 [6, {}&15, 18; 12, 8, 13], [7, 16, 19; 13, 9, 14], [8, 17, 11; 14, 10, 15], [9, 18, 12; 15, 2, 16], \\
& [10, 19, 13; 16, 3, 17]\bigr\}\\
\cup \bigl\{(2, {}&12, 14, 3, 11, 1), (3, 13, 15, 4, 12, 1), (4, 14, 16, 5, 13, 1), (5, 15, 17, 6, 14, 1),  \\
& (6, 16, 18, 7, 15, 1), (7, 17, 19, 8, 16, 1), (8, 18, 11, 9, 17, 1), (9, 19, 12, 10, 18, 1), \\
& (10, 11, 13, 2, 19, 1), (2, 3, 10, 4, 9, 5), (2, 6, 8, 7, 3, 4), (2, 7, 4, 5, 3, 8), \\
& (2, 10, 8, 4, 6, 9), (3, 6, 10, 5, 7, 9), (5, 6, 7, 10, 9, 8)\bigr\}
\end{align*}

\subsection{General constructions for multidecompositions}

\begin{lemma}\label{0mod6 decomp}
If $n \equiv 0 \pmod{6}$ with $n\geq 6$, then $K_n$ admits a $(C_6,\overline{C}_6)$-multide\-composition.
\end{lemma}
\begin{proof} 
Let $n=6x$ for some integer $x\geq 1$.  Note that $K_{6x}\cong \bigvee_{i=1}^{x}K_{6}$.  On each copy of $K_{6}$ place a $(C_{6},\overline{C}_{6})$-multidecomposition of $K_{6}$.  The remaining edges form edge-disjoint copies of $K_{6,6}$, which admits a $C_{6}$-decomposition by Theorem \ref{SotteauCorollary}.  Thus, we obtain the desired $(C_{6},\overline{C}_{6})$-multidecomposition of $K_{n}$.
\end{proof}

\begin{lemma}\label{1mod6 decomp}
If $n \equiv 1 \pmod{6}$ with $n\geq 13$, then $K_n$ admits a $(C_6,\overline{C}_6)$-multide\-composition.
\end{lemma}
\begin{proof} 
Let $n=6x+1$ for some integer $x\geq 2$.  The proof breaks into two cases.

\noindent \textbf{Case 1}: $x=2k$ for some integer $k\geq 1$.  Notice that $K_{12k+1}\cong K_{1}\vee\big(\bigvee_{i=1}^{k}K_{12}\big)$.  Each of the $k$ copies of $K_{13}$ formed by $K_{1}\vee K_{12}$ admit a $(C_6,\overline{C}_6)$-multide\-composition by Example \ref{K13}.  The remaining edges form edge-disjoint copies of $K_{12,12}$, which admits a $C_{6}$-decomposition by Theorem \ref{SotteauCorollary}.  Thus, we obtain the desired $(C_{6},\overline{C}_{6})$-multidecomposition of $K_{n}$.

\noindent \textbf{Case 2}: $x=2k+1$ for some integer $k\geq 2$.  Notice that $K_{12k+7}\cong K_{1}\vee K_{6}\vee\big(\bigvee_{i=1}^{k}K_{12}\big)$.  The single copy of $K_{19}$ formed by $K_{1}\vee K_{6}\vee K_{12}$ admits a $(C_{6},\overline{C}_{6})$-multidecomposition by Example \ref{K19}.  The remaining $k-1$ copies of $K_{13}$ formed by $K_{1}\vee K_{12}$ each admit a $(C_{6},\overline{C}_{6})$-multidecomposition by Example \ref{K13}.  The remaining edges form edge-disjoint copies of either $K_{6,12}$ or $K_{12,12}$.  Both of these graphs admit $C_{6}$-decompositions by Theorem \ref{SotteauCorollary}.  Thus, we obtain the desired $(C_{6},\overline{C}_{6})$-multidecomposition of $K_{n}$.
\end{proof}

\begin{lemma}\label{3mod6 decomp}
If $n \equiv 3 \pmod{6}$ with $n\geq 15$, then $K_n$ admits a $(C_6,\overline{C}_6)$ -multidecomposition.
\end{lemma}
\begin{proof} 
Let $n=6x+3$ for some integer $x\geq 2$.  The proof breaks into two cases.

\noindent \textbf{Case 1}: $x=2k$ for some integer $k\geq 1$.  
Notice that $K_{12k+3}\cong K_{1}\vee K_{14}\vee\big(\bigvee_{i=1}^{k-1}K_{12}\big)$.  The remainder of the proof is similar to the proof of Case 1 of Lemma \ref{1mod6 decomp} where the ingredients required are $C_{6}$-decompositions of $K_{12,12}$, and $K_{12,14}$, as well as $(C_6,\overline{C}_6)$-multidecompositions of $K_{13}$ and $K_{15}$.


\noindent \textbf{Case 2}: $x=2k+1$ for some integer $k\geq 1$.  Notice that $K_{12k+9}\cong K_{1}\vee K_{8}\vee\big(\bigvee_{i=1}^{k}K_{12}\big)$.  The remainder of the proof is similar to the proof of Case 2 of Lemma \ref{1mod6 decomp} where the ingredients required are $C_{6}$-decompositions of $K_{9}$ (which exists by Theorem \ref{C6spectrum}), $K_{8,12}$, and $K_{12,12}$, as well as a $(C_6,\overline{C}_6)$-multidecomposition of $K_{13}$.

\end{proof}

\begin{lemma}\label{4mod6 decomp}
If $n \equiv 4 \pmod{6}$ with $n\geq 16$, then $K_n$ admits a $(C_6,\overline{C}_6)$ -multidecomposition.
\end{lemma}
\begin{proof} 
Let $n=6x+4$ where $x\geq 2$ is an integer.  Note that $K_{6x+4}\cong K_{10} \vee \big(\bigvee_{i=1}^{x-1}K_{6}\big)$.  The remainder of the proof is similar to the proof of Case 2 of Lemma \ref{1mod6 decomp} where the ingredients required are $C_{6}$-decompositions of $K_{6,6}$ and $K_{6,10}$, a $\overline{C}_{6}$-decomposition of $K_{10}$ (which exists by Theorem \ref{ComplementSpectrum}), as well as a $(C_6$,$\overline{\textup{$C$}}_6)$-multidecomposition of $K_{6}$.

\end{proof}

Combining Lemmas \ref{0mod6 decomp}, \ref{1mod6 decomp}, \ref{3mod6 decomp}, and \ref{4mod6 decomp}, we have proven Theorem \ref{main decomposition theorem}.

\section{Maximum Multipackings}
Now we turn our attention to $(C_{6},\overline{C}_{6})$-multipackings in the cases where $(C_{6},\overline{C}_{6})$-~multidecompositions do not exist.  

\subsection{Small examples of maximum multipackings}
\begin{example}\label{K7 packing}
A maximum $(C_6,\overline{C}_6)$-multipacking of $K_{7}$ has a leave of cardinality 6.
\end{example}
Note that the number of edges used in a $(C_6,\overline{C}_6)$-multipacking of any graph must be a multiple of 3, since $\gcd(6,9)=3$. Since no $(C_6,\overline{C}_6)$-multidecomposition of $K_{7}$ exists the next possibility is a leave of cardinality 3.  However, the equation $18=6x+9y$ has no positive integer solutions.  Thus, the minimum possible cardinality of a leave is 6.  Let $V(K_{7}) = \{1,...,7\}$. The following is a $(C_6,\overline{C}_6)$-multipacking of $K_{7}$, with leave $\{\{1,7\}, \{2,7\}, \{3,7\}, \{4,7\}, \{5,7\}, \{6,7\}\}$.
$$\{[1,3,5;4,6,2], (1,2,3,4,5,6)\}$$
\begin{example}\label{K8}
A maximum $(C_6,\overline{C}_6)$-multipacking of $K_{8}$ has a leave of cardinality 1.
\end {example}
Let $V(K_{8}) = \{1,...,8\}$. The following is a $(C_6,\overline{C}_6)$-multipacking of $K_{8}$, with leave $\{3,6\}$.
\begin{align*}
&\{[2,5,7;4,1,8], (1,2,3,4,5,6), (1,3,5,8,6,7), (3,8,2,6,4,7)\}
\end{align*}

\begin{example}\label{K9 packing}
A maximum $(C_6,\overline{C}_6)$-multipacking of $K_{9}$ has a leave of cardinality 3.
\end{example}
Let $V(K_{9}) = \{1,...,9\}$. The following is a $(C_6,\overline{C}_{6})$-multipacking of $K_{9}$, with leave $\{\{2,4\},\{2,9\},\{4,9\}\}$.

$$\{[1,2,3;6,5,4], [1,4,7;9,8,3], [2,6,8;7,9,5], (1,5,3,6,7,8)\} $$

\begin{example}\label{K10 packing}
A maximum $(C_6,\overline{C}_6)$-multipacking of $K_{10}$ has a leave of cardinality 3.
\end{example}
A $(C_6,\overline{C}_6)$-multipacking of $K_{10}$ with a leave of cardinality 3 can be obtained by starting with a $\overline{C}_6$-decomposition of $K_{10}$.  Then remove three vertex-disjiont edges from one copy of $\overline{C}_{6}$, forming a $C_{6}$.  This gives us the desired $(C_6,\overline{C}_6)$-multipacking of $K_{10}$ where the three removed edge form the leave.

\begin{example}\label{K11}
A maximum $(C_6,\overline{C}_6)$-multipacking of $K_{11}$ has a leave of cardinality 1.
\end{example}
Let $V(K_{11}) = \{1,...,11\}$. The following is a $(C_6,\overline{C}_{6})$-multipacking of $K_{11}$, with leave $\{1,2\}$.
\begin{align*}
\{[1,7,{}&10;9,6,3], [1,5,6;4,10,2], [2,5,7;11,8,4], [1,3,11;8,2,9]\} \\
\cup \{&(3,4,9,10,6,8), (4,5,9,7,11,6), (3,5,11,10,8,7)\} 
\end{align*}


\begin{example}\label{K17}
A maximum $(C_6,\overline{C}_6)$-multipacking of $K_{17}$ has a leave of cardinality 1.
\end{example}
Let $V(K_{17}) = \{1,...,17\}$. The following is a $(C_6,\overline{C}_6)$-multipacking of $K_{17}$, with leave $\{1,10\}$.
\begin{align*}
\bigl\{[{}&2,3,5;7,8,1], [3,6,4;9,8,10],[2,4,9;6,5,7]\bigr\} \\
\cup \bigl\{&(2,12,5,10,11,14), (2,10,17,4,13,11), (4,7,13,14,5,15), (4,11,15,8,16,12), \\ 
&(1,15,14,16,5,17), (3,12,11,17,6,15), (1,2,16,7,14,4), (2,13,5,8,14,17), \\ 
&(7,15,10,13,9,17), (1,13,6,9,11,16), (1,9,12,7,3,11), (3,10,12,8,4,16), \\
&(3,13,16,6,12,14), (2,8,13,17,12,15), (6,10,16,15,9,14), (5,9,16,17,8,11), \\
&(1,3,17,15,13,12), (1,6,11,7,10,14)\bigr\}
\end{align*}

\subsection{General Constructions of maximum multipackings}

\begin{lemma}\label{2mod6 packing}
If $n \equiv 2 \pmod{6}$ with $n\geq 14$, then $K_n$ admits a $(C_6,\overline{C}_6)$-multi\-packing with leave cardinality 1.
\end{lemma}
\begin{proof}
Let $n=6x+2$ for some integer $x\geq 2$.  Notice that $K_{6x+2}\cong K_{2}\vee\big(\bigvee_{i=1}^{x}K_{6}\big)$.  Let $\{u, v\} = V(K_{2})$.  Each of the $x$ copies of $K_{8}$ formed by $K_{2}\vee K_{6}$ admit a $(C_6,\overline{C}_6)$-multipacking with leave cardinality 1 by Example \ref{K8}. Note that we can always choose the leave edge to be $\{u,v\}$ in each of these multipackings.  The remaining edges form edge disjoint copies of $K_{6,6}$, each of which admits a $C_{6}$-decomposition by Theorem \ref{SotteauCorollary}.  Thus, we obtain the desired $(C_{6},\overline{C}_{6})$-multipacking of $K_{n}$.

%
\end{proof}

\begin{lemma}\label{5mod6 packing}
If $n \equiv 5 \pmod{6}$ with $n\geq 11$, then $K_n$ admits a $(C_6,\overline{C}_6)$-multipacking with leave cardinality 1.
\end{lemma}

\begin{proof}
Let $n=6x+5$ for some integer $x\geq 1$.

\noindent \textbf{Case 1}: $x=2k$ for some integer $k\geq 1$.  Notice that $K_{12k+5}\cong K_{1}\vee K_{16}\vee\big(\bigvee_{i=1}^{k-1}K_{12}\big)$.  Each of the $k-1$ copies of $K_{13}$ formed by $K_{1}\vee K_{12}$ admit a $(C_6,\overline{C}_6)$-multidecomposition by Example \ref{K13}.  The copy of $K_{17}$ formed by $K_{1}\vee K_{16}$ admits a $(C_6,\overline{C}_6)$-multipacking with leave of cardinality 1 by Example \ref{K17}.  The remaining edges form edge disjoint copies of $K_{12,12}$ or $K_{12,16}$, each of which admits a $C_{6}$-decomposition by Theorem \ref{SotteauCorollary}.  Thus, we obtain the desired $(C_{6},\overline{C}_{6})$-multipacking of $K_{n}$.

\noindent \textbf{Case 2}: $x=2k+1$ for some integer $k\geq 1$.  Notice that $K_{12k+11}\cong K_{1}\vee K_{10} \vee\big(\bigvee_{i=1}^{k}K_{12}\big)$.  On each of the $k$ copies of $K_{13}$ formed by $K_{1}\vee K_{12}$ admit a $(C_6,\overline{C}_6)$-multidecomposition by Example \ref{K13}.  The copy of $K_{11}$ formed by $K_{1}\vee K_{10}$ admits a $(C_6,\overline{C}_6)$-multipacking with leave of cardinality 1 by Example \ref{K11}.  The remaining edges form edge disjoint copies of $K_{12,12}$ or $K_{10,12}$, each of which admits a $C_{6}$-decomposition by Theorem \ref{SotteauCorollary}.  Thus, we obtain the desired $(C_{6},\overline{C}_{6})$-multipacking of $K_{n}$.
\end{proof}

Combining Lemmas \ref{2mod6 packing} and \ref{5mod6 packing}, we have proven Theorem \ref{main packing theorem}.

\section{Minimum Multicoverings}
Now we turn our attention to minimum $(C_{6},\overline{C}_{6})$-multicoverings in the cases where $(C_{6},\overline{C}_{6})$-~multidecompositions do not exist.

\subsection{Small examples of minimum multicoverings}
\begin{example}\label{K7 covering}
A minimum $(C_6,\overline{C}_6)$-multicovering of $K_{7}$ has a padding of cardinality 6.
\end{example}
We first rule out the possibility of a minimum $(C_6,\overline{C}_6)$-multicovering of $K_{7}$ with a padding of cardinality 3.  The only positive integer solution to the equation $24=6x+9y$ is $(x,y)=(1,2)$.  In such a covering there would be one vertex left out of one of the copies of $\overline{C}_{6}$.  It would be impossible to use all edges at this vertex with the remaining copies of $C_{6}$ and $\overline{C}_{6}$.  Thus, the best possible cardinality of a padding is 6.  Let $V(K_{7}) = \{1,...,7\}$. The following is a minimum $(C_6,\overline{C}_{6})$-multicovering of $K_{7}$, with padding of $\{\{1,2\}, \{1,5\}, \{1,6\}, \{3,6\}, \{4,5\}, \{5,6\}\}$.
\begin{align*}
\{[1,2,3;6,5,4],(1,4,7,6,3,5), (1,6,2,4,5,7), (1,2,7,3,6,5)\}
\end{align*}

\begin{example}\label{K8 covering}
A minimum $(C_6,\overline{C}_6)$-multicovering of $K_{8}$ has a padding of cardinality 2.
\end{example}
Let $V(K_{8}) = \{1,...,8\}$. The following is a minimum $(C_6,\overline{C}_{6})$-multicovering of $K_{8}$, with padding of $\{\{1,8\}, \{3,5\}\}$.
\begin{align*}
\{[1,2,8;4,3,5], [1,5,6;3,7,8], (1,7,2,6,4,8), (2,4,7,6,3,5)\}
\end{align*}

\begin{example}\label{K9 covering}
A minimum $(C_6,\overline{C}_6)$-multicovering of $K_{9}$ has a padding of cardinality 3.
\end{example}

A $(C_6,\overline{C}_6)$-multicovering of $K_{9}$ with a padding of cardinality 3 can be obtained by starting with a $C_6$-decomposition of $K_{9}$.  One copy of $C_{6}$ can be transformed into a copy of $\overline{C}_{6}$ by adding the edges in a 1-factor on the vertices in a copy of $C_{6}$.  This gives us the desired $(C_6,\overline{C}_6)$-multicovering of $K_{9}$ where the three added edges form the padding.

\begin{example}\label{K10 covering}
A minimum $(C_6,\overline{C}_6)$-multicovering of $K_{10}$ has a padding of cardinality 3.
\end{example}

A $(C_6,\overline{C}_6)$-multicovering of $K_{10}$ with a padding of cardinality 3 can be obtained by starting with a $\overline{C}_{6}$-decomposition of $K_{10}$.  One copy of $\overline{C}_{6}$ can be transformed into two copies of $C_{6}$ by carefully adding three edges.  This gives us the desired $(C_6,\overline{C}_6)$-multicovering of $K_{10}$ where the three added edges form the padding.

\begin{example}\label{K11 covering}
A minimum $(C_6,\overline{C}_6)$-multicovering of $K_{11}$ has a padding of cardinality 2.
\end{example}
Let $V(K_{11}) = \{1,...,11\}$. The following is a minimum $(C_6,\overline{C}_{6})$-multicovering of $K_{11}$, with padding of $\{\{3,4\}, \{8,11\}\}$.
\begin{align*}
\{[1,2,{}&11;6,5,7], [1,3,5;10,2,9], [4,6,10;7,9,8]\} \\
\cup \bigl\{&(3,4,5,8,11,6), (1,8,2,7,3,9), (2,4,9,11,8,6), (1,4,3,11,10,7), \\ 
&(3,8,4,11,5,10) \bigr\}
\end{align*}

\begin{example}\label{K17 covering}
A minimum $(C_6,\overline{C}_6)$-multicovering of $K_{17}$ has a padding of cardinality 2.
\end{example}
Let $V(K_{17}) = \{1,...,17\}$.  Apply Theorem \ref{C6spectrum} and let $\mathcal{B}_{1}$ be a $C_{6}$-decomposition on the copy of $K_{9}$ formed by the subgraph induced by the vertices $\{9,\dots,17\}$.  Apply Theorem \ref{SotteauCorollary} and let $\mathcal{B}_{2}$ be a $C_{6}$-decomposition of the copy of $K_{6,8}$ formed by the subgraph of $K_{17}$ with vertex bipartition $(A,B)$ where $A=\{1,\dots,8\}$ and $B=\{12,\dots,17\}$.  The following is a minimum $(C_6,\overline{C}_{6})$-multicovering of $K_{17}$, with padding of $\{\{3,5\}, \{7,8\}\}$.  
\begin{align*}
\mathcal{B}_{1}&\cup\mathcal{B}_{2}\cup\{[1,2,3;6,5,4], [1,4,8;7,2,6]\} \\
\cup \bigl\{&(1,5,7,8,3,9), (1,10,3,7,4,11), (2,8,7,11,6,9) \bigr\}\\
\cup \bigl\{&(5,11,8,9,7,10), (3,5,9,4,10,6), (2,11,3,5,8,10) \bigr\}
\end{align*}

\subsection{General constructions of minimum multicoverings}

\begin{lemma}\label{2mod6 covering}
If $n \equiv 2 \pmod{6}$ with $n\geq 8$, then $K_n$ admits a minimum $(C_6,\overline{C}_6)$-multicovering with a padding of cardinality 2.
\end{lemma}
\begin{proof}
Let $n=6x+2$ for some integer $x\geq 1$.   Notice that $K_{6x+2}\cong K_{8}\vee\big(\bigvee_{i=1}^{x-1}K_{6}\big)$.  Each of the $x-1$ copies of $K_{6}$ admit a $(C_6,\overline{C}_6)$-multidecomposition by Lemma \ref{0mod6 decomp}.  The copy of $K_{8}$ admits a $(C_6,\overline{C}_6)$-multicovering with a padding of cardinality 2 by Example \ref{K8 covering}.  The remaining edges form edge disjoint copies of $K_{6,6}$ or $K_{6,8}$, each of which admit a $C_{6}$-decomposition by Theorem \ref{SotteauCorollary}.  Thus, we obtain the desired $(C_{6},\overline{C}_{6})$-multicovering of $K_{n}$. 
\end{proof}

\begin{lemma}\label{5mod6 covering}
If $n \equiv 5 \pmod{6}$ with $n\geq 11$, then $K_n$ admits a minimum $(C_6,\overline{C}_6)$-~multicovering with a padding of cardinality 2.
\end{lemma}
\begin{proof}
Let $n=6x+5$ for some integer $x\geq 1$.  The proof breaks into two cases.

\noindent \textbf{Case 1}: $x=2k$ for some integer $k\geq 1$.  Notice that $K_{12k+5}\cong K_{1}\vee K_{4}\vee\big(\bigvee_{i=1}^{k}K_{12}\big)$.  One copy of $K_{17}$ is formed by $K_{1}\vee K_{4}\vee K_{12}$, and admits a $(C_6,\overline{C}_6)$-multicovering with a padding of cardinality 2 by Example \ref{K17 covering}.  The $k-1$ copies of $K_{13}$ formed by $K_{1}\vee K_{12}$ admit a $(C_6,\overline{C}_6)$-multidecomposition by Example \ref{K13}.  The remaining edges form edge disjoint copies of $K_{12,12}$ or $K_{4,12}$, each of which admits a $C_{6}$-decomposition by Theorem \ref{SotteauCorollary}.  Thus, we obtain the desired $(C_{6},\overline{C}_{6})$-multicovering of $K_{n}$.

\noindent \textbf{Case 2}: $x=2k+1$ for some integer $k\geq 1$.  Notice that $K_{12k+11}\cong K_{1}\vee K_{4}\vee K_{6} \vee\big(\bigvee_{i=1}^{k}K_{12}\big)$.  One copy of $K_{11}$ is formed by $K_{1}\vee K_{4}\vee K_{6}$, and admits a $(C_6,\overline{C}_6)$-multicovering with a padding of cardinality 2 by Example \ref{K11 covering}.  The $k$ copies of $K_{13}$ formed by $K_{1}\vee K_{12}$ admit a $(C_6,\overline{C}_6)$-multidecomposition by Example \ref{K13}.  The remaining edges form edge disjoint copies of $K_{12,12}$, $K_{4,12}$, or $K_{6,12}$, each of which admits a $C_{6}$-decomposition by Theorem \ref{SotteauCorollary}.  Thus, we obtain the desired $(C_{6},\overline{C}_{6})$-multicovering of $K_{n}$.
\end{proof}

Combining Lemmas \ref{2mod6 covering} and \ref{5mod6 covering}, we have proven Theorem \ref{main covering theorem}.

\section{Final notes}
The cardinalities of the leaves of maximum $(C_{6},\overline{C}_{6})$-multipackings and paddings of minimum $(C_{6},\overline{C}_{6})$-multicoverings of $K_{n}$ have been characterized.  It is still an open problem to characterize the structure of those leaves and paddings.

We would like to thank Mark Liffiton and Wenting Zhao for finding $(C_{6},\overline{C}_{6})$-multidecom\-positions of $K_{11}$ and $K_{17}$ using the MiniCard solver.  MiniCard source code is available at \url{https://github.com/liffiton/minicard}.

\end{document}